\documentclass[english,10pt, a4paper]{amsart}

\usepackage{amsmath}
\usepackage{amssymb}
\usepackage{times}
\usepackage{mathrsfs}
\usepackage{graphicx}

\begin{document}
\numberwithin{equation}{section}

\def\1#1{\overline{#1}}
\def\2#1{\widetilde{#1}}
\def\3#1{\widehat{#1}}
\def\4#1{\mathbb{#1}}
\def\5#1{\frak{#1}}
\def\6#1{{\mathcal{#1}}}

\newcommand{\bC}{{\mathbb{C}}}
\newcommand{\bP}{{\mathbb{P}}}
\newcommand{\bF}{{\mathbb{F}}}
\newcommand{\bZ}{{\mathbb{Z}}}
\newcommand{\bR}{{\mathbb{R}}}
\newcommand{\si}{{\sigma}}
\newcommand{\D}{{\delta}}
\newcommand{\al}{{\alpha}}
\newcommand{\la}{{\lambda}}
\newcommand{\La}{{\Lambda}}
\newcommand{\G}{{\Gamma}}
\newcommand{\ga}{{\gamma}}
\newcommand{\W}{{W_d^r(Z)}}
\newcommand{\tW}{{\widetilde{W}}}
\newcommand{\tX}{{\widetilde{X}}}
\newcommand{\tp}{{\widetilde{p}}}
\newcommand{\tq}{{\widetilde{q}}}
\newcommand{\tS}{{\widetilde{\Si}}}
\newcommand{\f}{{\varphi}}
\newcommand{\Oc}{{\mathcal{O}}}

\newcommand{\s}{{\mathcal{S}}}
\newcommand{\pP}{{\mathcal{P}}}
\newcommand{\M}{{\mathcal{M}}}
\newcommand{\g}{{\mathfrak{g}}}
\newcommand{\X}{{\mathcal{X}}}
\newcommand{\T}{{\mathcal{T}}}
\newcommand{\Li}{{\mathcal{L}}}
\newcommand{\grad}{{\underline{\rm de}{\rm g}}}
\newcommand{\de}{{\underline{d}}}
\newcommand{\cu}{{C_1}}
\newcommand{\cd}{{C_2}}
\newcommand{\Pic}{{{\rm Pic}}}
\newcommand{\Pbar}{{\overline{P^d_X}}}
\newcommand{\IM}{{{\rm Im}}}
\newcommand{\ov}{\overline}
\newcommand{\un}{\underline}
\newcommand{\wt}{\widetilde}
\newcommand{\wh}{\widehat}

\newtheorem{teo}{Theorem}[subsection]
\newtheorem{lemma}[teo]{Lemma}
\newtheorem{defi}[teo]{Definition}
\newtheorem{prop}[teo]{Proposition}
\newtheorem{cor}[teo]{Corollary}
\newtheorem{pb}[teo]{Problem}
\newtheorem{parte}[]{Part}

\newtheorem{exer}{Exercise}[section]
\newtheorem{example}{Example}[section]
\newtheorem{remark}[teo]{Remark}

\title{Compactifying the image of the Abel map}

\author{Silvia Brannetti}
\address{Dipartimento di Matematica,
Universit\`a Roma Tre,
Largo S. Leonardo Murialdo 1,
00146 Roma (Italy)}
\email{brannett@mat.uniroma3.it}

\begin{abstract}
Let $\al_X^{\de}$ be the Abel map of multidegree $\de$ of a singular curve $X$ of genus $g$. We describe the closure of $\IM\al_X^\de$ inside Caporaso's compactified Jacobian $\ov{P_X^d}$ for irreducible curves, curves of compact type and binary curves.
\end{abstract}

\maketitle


\section{Introduction}

Abel maps have been studied since the nineteenth century, starting from smooth curves:
given a smooth projective curve $C$ of genus $g$ and a natural number $d\geq 1$, we can consider the product $C^d$ and define the \textit{Abel map of degree $d$}
$$\begin{array}{cccc}
\al_C^d:& C^d&\longrightarrow& \Pic^dC\\
&(p_1,\ldots,p_d)&\mapsto &\Oc_C(\sum\limits_{i=1}^{d}p_i);
\end{array}$$
it is a regular map, and in degree $1$ it is injective when $g\geq 1$. In particular, when $d=1$ it gives the Abel-Jacobi embedding of $C$ into $\Pic^1 C\cong J(C)$, its Jacobian. In the smooth case the image $\IM\al_C^d$ of the Abel map coincides with a Brill-Noether variety, $W_d(C)$, defined as

$$W_d(C)=\{L\in{\rm Pic}^d C\mbox{ s.t. }h^0(C,L)> 0\}.$$

We recall that $\dim W_d(C)=\min\{d,g\}$. A natural problem is to extend the Abel maps to singular curves in such a way that they have a geometric meaning.

If $X$ is a singular curve we can still define Abel maps, but if we proceed as in the smooth case we will get a rational map, since the singular points are not Cartier divisors. We consider
the decomposition of $X$ in irreducible components, $X=C_1\cup\ldots\cup C_\ga$, and set $\dot X:=X\setminus X^{\rm sing}$, where $X^{\rm sing}$ is the set of nodes of $X$, and $\dot C_i=C_i\cap\dot X$. Then, let $\de=(d_1,\ldots,d_\ga)$ be a multidegree with $d_i\geq 0$ for any $i$, and $$\dot X^\de:=\dot C_1^{d_1}\times\cdots\times\dot C_\ga^{d_\ga};$$
we notice that $\dot X^\de$ is a smooth irreducible variety of dimension $d=|\de|$, open and dense in $X^\de:=C_1^{d_1}\times\cdots\times C_\ga^{d_\ga}$. We set
$$\begin{array}{cccc}
\al_X^\de:& \dot X^\de&\longrightarrow& \Pic^\de X\\
&(p_1,\ldots,p_d)&\mapsto &\Oc_X(\sum\limits_{i=1}^{d}p_i),
\end{array}$$
and we call it the \textit{Abel map of multidegree $\de$}; it is a regular map.
Abel maps for integral curves have been studied by Altman and Kleiman in \cite{bib:AltKl}, and later on in \cite{bib:EGK1}, \cite{bib:EGK2}, \cite{bib:EK}. We notice that the completion of Abel maps for integral curves was a major step to prove \textit{autoduality} of the compactified Jacobian (\cite{bib:EGK2}). This is an important property connected with the study of the fibers of the Hitchin fibration for $GL(n)$ (\cite{AKI}, \cite{Ngo}).

For reducible curves, the problem of completing the Abel maps is open with a few exceptions as we shall explain. As it is well known, the non separatedness of the Picard functor, together with combinatorial hurdles, make the case of reducible curves much more complex.
The first step in this direction was taken by Caporaso and Esteves in \cite{bib:CapEst}, where they construct Abel maps of degree $1$ for stable curves. For the completion of the map they use as target space the compactified Picard scheme $\ov{P_X^1}$ constructed in \cite{bib:compJac}; they prove that there exists a regular map
$$\ov{\al_X^1}:X\longrightarrow \ov{P_X^1}$$
which extends $\al_X^1$. However, they do not describe explicitly the closure of the image of the completion of the map. It is interesting to notice that they consider stable curves as limits of smooth ones, approaching this way the study of Abel maps for families of curves. In this setting, the completion of $\al_X^1$ can be viewed as a specialization to the singular fiber of the Abel maps of the smooth fibers. It is important to point out that the map $\ov{\al_X^1}$ turns out to be independent of the smoothing of $X$. This last aspect has been studied in \cite{bib:naturality}, where the author characterizes in purely combinatorial terms the stable curves having \textit{natural} Abel maps: again looking at a stable curve $X$ as limit of smooth ones, an Abel map for $X$ is natural if it doesn't depend on the choice of the smoothing.

Further improvements have been achieved for Gorenstein curves by Caporaso, Coelho and Esteves in \cite{bib:gorenstein} using torsion free sheaves, and by Coelho and Pacini in \cite{bib:CoelTesi} and \cite{bib:CoelPac}, where, respectively, they construct Abel maps of degree $2$ for curves with two components and two nodes, and in any degree for curves of compact type. So in all other cases this problem remains open.

On the other hand the situation is better understood in case $d=g-1$ in \cite{bib:theta}: if $X$ is a nodal connected curve of genus $g$, denote by $A_\de(X)$  the closure of $\IM\al_X^\de$ inside $\Pic^\de X$. Let
$$W_\de(X):=\{L\in\Pic^\de X:h^0(L)>0\};$$
in Theorem 3.1.2. the author proves that if $\de$ is a stable multidegree such that $|\de|=g-1$, then $$A_\de(X)=W_\de(X),$$ and hence that the Brill-Noether variety $W_\de(X)$ is irreducible. Let $\overline{P^{g-1}_X}$ be the compactified Jacobian in degree $g-1$; it has a polarization given by the Theta divisor $\Theta(X)$, and the pair $(\overline{P^{g-1}_X},\Theta(X))$ is a semiabelic stable pair as in \cite{bib:alexeev1}. It turns out that the varieties $A_\de(X_S)=W_\de(X_S)$, where $X_S$ is a partial normalization of $X$ at a set $S$ of nodes, are the sets which give a stratification of $\Theta(X)$ (see Theorem 4.2.6. in \cite{bib:theta}).

The goal of this paper is to generalize this stratification in lower degree and give a characterization of the closure of the image of the Abel map of multidegree $\de$ for some classes of nodal curves, inside the compactified Picard variety $\Pbar$ constructed in \cite{bib:compJac}. We recall that in this construction every point of $\Pbar$ corresponds to a pair $(\wh X_S,\wh M_S)$ where $\wh X_S$ is the blow up of $X$ at a set $S$ of nodes of $X$, and $\wh M_S$ is a \textit{balanced} line bundle (see below) of multidegree $\de$ on $\wh X_S$ up to equivalence. So our question can be posed in the following way: which points of $\Pbar$ are limits of effective Weil divisors on $X$?

We will study the following cases: irreducible curves on the one hand, and two types of reducible curves, namely curves of compact type and binary curves. Curves of compact type have the advantage and the special property that the generalized Jacobian is compact. Binary curves are nodal curves made of two smooth rational components meeting at $g+1$ points. They form a remarkable class of reducible curves since they present the basic problems as all reducible curves, yet simpler combinatorics. Indeed, they have been used in the past as test cases for results later generalized to all stable curves, see for instance \cite{bib:binary},\cite{bib:Bruno}.

In order to answer our question, let $X_S$ be a partial normalization of a nodal curve $X$ at a set $S$ of nodes. We define the set

$$W_{\de_S}^{+}(X_S)=\{L\in\Pic^{\de_S} X_S:h^0(Z,L|_Z)>0\mbox{ for all subcurves }Z\subseteq X_S\},$$
and consider the union of
the $W_{\de_S}^{+}(X_S)$ when $S$ varies among the subsets of $X^{\rm sing}$ and $\de_S$ is the restriction to $X_S$ of a balanced multidegree ${\widehat \de}_S$ on the partial blow up $\wh X_S$. Similarly to \cite{bib:theta}[Theorem 4.2.6], we define
$$\wt W_d(X):=\bigsqcup\limits_{\emptyset\subset S\subset X^{\rm sing} \atop \wh \de_S\in B^{^{\geq0}}_{d}(\wh X_S)}W_{\de_S}^{+}(X_S),$$
where $B^{^{\geq0}}_{d}(\wh X_S)$ is the set of strictly balanced multidegrees $\widehat \de_S\geq\un 0$ on $\wh X_S$ such that $|\widehat \de_S|=d$, and $\de_S={\widehat \de}_S{|_{X_S}}$.

In section \ref{irrcurves} we study directly the closure inside $\Pbar$ of $A_d(X)$, and we prove that $\overline{A_d(X)}=\overline{W_d(X)}$ giving a description of it in terms of the Brill-Noether varieties $W^0_{d-\D_S}(X_S)$ where $X_S$ is the normalization of $X$ at a set of nodes $S$, and $\delta_S=\sharp S$.

In section \ref{redcurves} we turn our attention to reducible curves: we describe the structure of the varieties $W_\de(X)$ for curves of compact type, which is quite natural, and in the last part we develop the study of $A_\de(X)$ and its closure inside $\Pbar$ for binary curves. We characterize it in terms of the varieties $W_{\de_S}(X_S)$.
If $X$ is a binary curve of genus $g$ and $1\leq d\leq g-1$, we prove that the closure inside $\Pbar$ of the union of the varieties $A_\de(X)$ as $\de$ varies among balanced multidegrees on $X$, is exactly $\wt W_d(X)$. In other words, we define
$$\ov{A_d(X)}:=\ov{\bigcup_{\de\in B^{^{\geq0}}_d(X)}A_\de(X)}\subset \Pbar,$$

then the main theorem states that

\begin{equation}
\wt W_d(X)=\ov{A_d(X)}\subset\Pbar.
\end{equation}
Finally we study the simpler case when $d=1$ giving a characterization of the closure of the image of the Abel map for all the stable curves such that the set $B^{\geq 0}_1(X)$ of strictly balanced multidegrees $\de\geq\un 0$ is nonempty, i.e. the so called \textit{$d$-general} curves.

\section*{Acknowledgements}
It is a pleasure to thank my Ph.D. advisor Lucia Caporaso for introducing me this problem and for her guidance and suggestions. I also want to thank Eduardo Esteves for useful remarks and Edoardo Ballico and Claudio Fontanari for encouraging discussions.

\section{Notation}
Let us recall some basic facts about the construction in \cite{bib:compJac} that we will use in what follows. We work over an algebraically closed field $k$. Throughout the paper a \textit{curve} will be a reduced projective variety of pure dimension $1$ over $k$. Moreover, we will deal with nodal curves, although some statements are more general. Let then $X$ be a nodal curve, and let $X^\nu\stackrel{\nu}{\longrightarrow}X$ be its normalization; if $X^\nu=\sqcup_{i=1}^\gamma C^\nu_i$ is the decomposition of $X^\nu$ into smooth components of genus $g_i$ for every $i=1,\ldots,\ga$, then the arithmetic genus of $X$ is $g=\sum_{i=1}^\ga g_i+\delta-\ga+1$. If $Z$ is a subcurve of $X$ of genus $g_Z$ and $Z^c=\ov{X\setminus Z}$, we will denote by $\delta_Z=\sharp Z\cap Z^c$ and if $\omega_X$ is the dualizing sheaf of $X$, we set $\deg_Z \omega_X=\deg\omega_X|_Z=2g_Z-2+\delta_Z$.

A curve $X$ of genus $g\geq 2$ is said to be \textit{stable} if it is connected and if every component $E\cong\bP^1$ is such that $\delta_E\geq 3$, which is equivalent to saying that the curve has finite automorphism group. By a \textit{quasistable} curve we mean a connected curve $X$ such that every subcurve $E\cong\bP^1$ has $\delta_E\geq 2$ and the ones with $\delta_E=2$, i.e. the exceptional components, don't intersect. If $S$ is a set of nodes of a stable curve $X$, throughout the paper we will denote by $X_S$ the normalization of $X$ at the nodes in $S$, and by $\wh X_S$ the quasistable curve obtained by ``blowing up'' $X$ at $S$. In what follows we will often call $\wh X_S$ a partial blow up of $X$. Obviously $X_S$ is the complement in $\wh X_S$ of all the exceptional components.

In \cite{bib:compJac} Caporaso constructs a compactification $\ov{ P}_{d,g}\rightarrow\ov{M}_g$ of the universal Picard variety, such that the fiber over a smooth curve $X$ of genus $g\geq 2$ is its Picard variety $\Pic^d X$, whereas if $X$ is a stable curve in $\ov{M}_g$, then the fiber over it is $\Pbar$, a connected and projective scheme, which has a meaningful description in terms of line bundles on the partial blowups of $X$.

Indeed, let $X$ be a quasistable curve of genus $g$ and $L\in\Pic^d X$; we denote the multidegree of $L$ by
$$\de=(d_1,\ldots,d_\ga),$$
where, if $X=\bigcup_{i=1}^\ga C_i$ is the decomposition of $X$ in irreducible components, we have $d_i=\deg L|_{C_i}$ and $d=|\de|$. We say that $\de$ is \textit{balanced} if for any connected subcurve $Z$ of $X$ we have that
\begin{equation}\label{BI}
d\frac{w_Z}{2g-2}-\frac{\delta_Z}2\leq d_i\leq d\frac{w_Z}{2g-2}+\frac{\delta_Z}2,
\end{equation}
where $w_Z=\deg_Z \omega_X$, and for any exceptional component $E$ of $X$ we have $L|_E=\Oc_E(1)$.

We say $\de$ is \textit{strictly balanced} if strict inequalities hold in (\ref{BI}) for every $Z\varsubsetneq X$ such that $Z\cap Z^c\not\subset X_{\rm exc}$, where $X_{\rm exc}$ is the subcurve of the exceptional components of $X$ (see \cite{bib:Néron}). We will denote by $\ov{B_d(X)}$ the set of balanced multidegrees on $X$, and by $B_d(X)$ its subset of strictly balanced ones.

We are going to introduce the scheme $\Pbar$ by looking at its stratification; so let $X$ be a stable curve of genus $g\geq 2$, then, for any $d$, $\Pbar$ is a connected, reduced scheme of pure dimension $g$, such that
\begin{equation}\label{strata}
\Pbar=\coprod_{\emptyset\subset S\subset X^{\rm sing} \atop \de\in B_{d}(\widehat X_S)}P_S^{\de},
\end{equation}
where $P_S^{\de}\cong\Pic^{\de_S}X_S$, $X_S\subset\wh X_S$ as above, and $\de_S=\de|_{X_S}$. In particular, the points in $\Pbar$ are in one-to-one correspondence with equivalence classes of strictly balanced line bundles. Any such class is determined by $S$ and by $M\in\Pic X_S$. Hence a point of $\Pbar$ can be denoted by $[M,S]$, where if $\wh M_S$ is a class of line bundles in $B_{d}(\widehat X_S)$, then $M:=\wh M_S|_{X_S}$, and, by construction, when restricted to every exceptional component of $\wh X_S$, $\wh M_S$ is equal to $\Oc(1)$.

A node $n$ of $X$ is said to be \textit{separating} if $X\setminus\{n\}$ is not connected; we denote by $X^{\rm sing}$ the set of nodes of $X$, and by $X_{\rm sep}$ the subset of separating nodes.

Let $\nu_S:X_S\rightarrow X$ be the normalization of $X$ at the nodes in $S$. It induces the pullback map $\nu_S^*:\Pic^\de X\rightarrow\Pic^\de X_S$; if $M\in\Pic^\de X_S$, we denote by $F_M(X)$ the fiber of $\nu_S^*$ over $M$, and by $W_M(X)$ the intersection $F_M(X)\cap W_\de(X)$.

\section{Irreducible curves}\label{irrcurves}

Let $X$ be an irreducible nodal curve of genus $g$ and, for $d\geq 1$, consider the Brill-Noether variety $W_d(X)$. As a subvariety of ${\rm Pic}^dX$, we are interested in studying its closure $\overline{W_d(X)}$ in the compactified Picard Variety $\overline{P^d_X}$, using the description given in \cite{bib:compJac}. It will turn out that $\overline{W_d(X)}$ is strongly related to the image of the Abel map, that we are going to define.
Let $\dot{X}:=X\setminus X^{\rm sing}$ be the smooth locus of $X$; since $X$ is irreducible, we have that $\dot{X}^d$ is a smooth irreducible variety of dimension $d$, open and dense in $X^d$. Now, for $d\geq 1$, let
$$\begin{array}{cccc}
\al_X^d:& \dot{X}^d&\longrightarrow& \Pic^dX\\
&(p_1,\ldots,p_d)&\mapsto &\Oc_X(\sum\limits_{i=1}^{d}p_i);
\end{array}$$
we call $\al_X^d$ the \textit{Abel map of degree} $d$. It is a regular map, and obviously $\al_X^d(\dot{X}^d)\subset W_d(X)$. We denote by $A_d(X)$ the closure of $\al_X^d(\dot{X}^d)$ in $\Pic^dX$; of course $A_d(X)\subset W_d(X)$.
Let us now introduce the following set
$$\widetilde W_d(X):=\{[M,S]\in\Pbar\mbox{ s.t. }h^0(\widehat X_S,\widehat M_S)> 0\},$$
where $S\subset X^{\rm sing}$ with $\D_S:=\sharp S$, $\widehat X_S=X_S\cup\cup_{i=1}^{\D_S}E_i$ is the blow up of $X$ at the nodes of $S$, and, as we introduced in the previous section, $\wh M_S$ is a class of line bundles in $B_{d}(\widehat X_S)$ such that its resctrictions to the components of $\widehat X_S$ are
$$\widehat M_S|_{X_S}=:M, \quad \widehat M_S|_{E_i}=\Oc(1)\mbox{ for any }i=1,\ldots\D_S.$$
Let us observe that since $h^0(\widehat X_S,\widehat M_S)=h^0(X_S,M)$ (see \cite{bib:theta}[Lemma 4.2.5]), we have:
$$\widetilde W_d(X)=\{[M,S]\in\Pbar\mbox{ s.t. }h^0(X_S,M)> 0\},$$
which is in turn equivalent to:
$$\widetilde W_d(X)\cong\bigsqcup\limits_{S\subset X^{\rm sing}}W_{d-\D_S}(X_S).$$

\begin{teo}\label{Ad(X)}
Let X be an irreducible curve of genus $g\geq 1$ with $\D$ nodes. Then for any $d\geq 1$ we have:
\begin{itemize}
\item[(i)]$A_d(X)=W_d(X)$, hence $W_d(X)$ is irreducible and $\dim W_d(X)=\min\{d,g\}$,
\item[(ii)]$\overline{A_d(X)}=\overline{W_d(X)}=\widetilde{W}_d(X)\subset\Pbar.$
\end{itemize}
\end{teo}
\begin{proof}
We start by assuming that $X$ has only one node $n$, and its normalization is $\nu:X_n\rightarrow X$, with $\nu^{-1}(n)=\{p,q\}$. Let us consider the regular dominant map
$$\begin{array}{cccc}
    \rho: & W_d(X) & \rightarrow & W_d(X_n) \\
     & L & \mapsto & \nu^*(L);
  \end{array}
$$
for any $M\in\IM\rho$ we denote by $W_M(X)=\rho^{-1}(M)$, the fiber of $\rho$. We recall that $W_M(X)\subset F_M(X)$, where $F_M(X)\cong k^*$ is the fiber of the pullback map $\nu^*:\Pic^dX\rightarrow {\rm Pic}^dX_n$.
The cardinality of the fibers $W_M(X)$ is at least $0$, so, since $\dim W_d(X_n)=d$, it follows that $\dim W_d(X)\leq d$; moreover $A_d(X)$ is irreducible of dimension $d$, hence we have that $A_d(X)$ is an irreducible component of $W_d(X)$.
We want to prove that for any $M\in\IM\rho$, $W_M(X)\subset\ov{A_d(X)}$, so that $A_d(X)\subset W_d(X)\subset\ov {A_d(X)}$ implies that $W_d(X)=A_d(X)$ and $\ov {W_d(X)}=\ov {A_d(X)}$.
We are now going to analyze all the possible cases.
\begin{itemize}
\item[(1)] $M\in\IM\rho$ with $h^0(X_n,M)=1$ and $$h^0(X_n,M(-p))=h^0(X_n,M(-q))=h^0(X_n,M)-1.$$ Then by \cite[Lemma 2.2.3]{bib:theta}, $W_M(X)=\{L_M\}$ with $L_M\in\IM\al_X^d$.
\item[(2)]$M\in\IM\rho$ with $h^0(X_n,M)\geq 2$ and $$h^0(X_n,M(-p))=h^0(X_n,M(-q))=h^0(X_n,M)-1.$$ We are going to show that there exist two points in $\ov {F_M(X)}\subset\Pbar$ which are contained in $\ov {A_d(X)}$. Indeed, $$\ov {F_M(X)}\setminus F_M(X)=\{[M(-p),n],[M(-q),n]\}.$$ Let us take $[M(-p),n]$; by \cite[Lemmas 2.2.3, 2.2.4]{bib:theta} there exists $L\in{\rm Pic}^{d-1}X$ such that $\nu^*(L)=M(-p)$ and $L\in\IM\al_X^{d-1}$. Let now $p_t\in\dot{X}$ be a moving point specializing to the node, i.e.such that $p_t\stackrel{t\rightarrow 0}{\longrightarrow}n$. Of course $L(p_t)\in\IM\al_X^d$, and $L(p_t)\rightarrow [M(-p),n]$ as $t\rightarrow 0$. Then $[M(-p),n]\in \ov{A_d(X)}$. The same holds for $[M(-q),n]$, so we have that $$\ov {F_M(X)}\setminus F_M(X)\subset\ov{A_d(X)}.$$
\item[(3)]$M\in\IM\rho$ with $h^0(X_n,M)=1$ and $$h^0(X_n,M(-p))=h^0(X_n,M(-q))=h^0(X_n,M).$$ Again we want to prove that $\ov {F_M(X)}\setminus F_M(X)\subset\ov{A_d(X)}$; so let $M'$ be a line bundle on $X_n$ not supported on either $p$ or $q$ such that $M=M'(hp+kq)$; then $M'$ is as in (1) and $\deg M'=d'$ with $d'=d-(h+k)$. Let us consider $[M(-p),n]\in\ov{F_M(X)}$, then $M(-p)=M'(h'p+kq)$, where $h'=h-1$. We choose a moving point $p_t$ on $X_n$ specializing to $p$ as $t$ goes to $0$, and a moving point $q_t$ on $X_n$ such that $q_t$ specializes to $q$. Now fix $t$, and take the line bundle $M_t'':=M'(h'p_t+kq_t)$ on $X_n$; by case (1), there exists $L_t''\in\IM\al_X^{d-1}$ such that $\nu^*(L_t'')=M''_t$. We consider now one moving point $p_u\in\dot{X}$, such that $\nu^*(p_u)$ on $X_n$ specializes to $p$ when $u\rightarrow 0$. As well as we saw in case (2), $L_t''(p_u)\in\IM\al_X^d$ specializes to $[M_t'',n]$ as $u\rightarrow 0$. Hence $[M_t'',n]\in\ov{A_d(X)}$. Now let $t\rightarrow 0$: we see that, by construction, $[M_t'',n]\rightarrow [M(-p),n]$, hence $[M(-p),n]\in\ov{A_d(X)}$. Using the same argument, we get that $[M(-q),n]\in\ov{A_d(X)}$ as well.
\item[(4)]$M\in\IM\rho$ with $h^0(X_n,M)\geq 2$ and either $p$ or $q$ as base point. Choose, say, $p$ as base point, i.e. $h^0(X_n,M(-p))=h^0(X_n,M)=h^0(X_n,M(-q))+1$. Then there exists $M'\in {\rm Pic}^{d'}X_n$, with $M=M'(hp)$, $d'=d-h$, and $M'$ not supported on either $p$ or $q$ up to move the support away. We notice that $M(-p)=M'(h'p)$ with $h'=h-1$, so, as before, we perform a double specialization to show that $[M(-p),n]\in\ov{A_d(X)}$. Concerning $[M(-q),n]$, we have that $M(-q)=M'(hp-q)=:M''(hp)$ for a suitable $M''\in{\rm Pic}^{d'-1}X_n$. Moreover, since $p$ is a base point of $M''(hp)$, $h^0(X_n,M'')\geq 1$. We take again a moving point $p_t$ on $X_n$ specializing to $p$, and a $p_u$ on $X$ such that $\nu^*(p_u)$ specializes to $p$ on $X_n$. We fix $t$ and denote $M_t'':=M''(hp_t)$, then by \cite[Lemmas 2.2.3,2.2.4]{bib:theta} there exists $L_t''$ contained in $\IM\al_X^{d'-1}$ such that $\nu^*(L_t'')=M_t''$. We take $L_t''(p_u)$; letting $u\rightarrow 0$ we get that $L_t''(p_u)\rightarrow [M_t'',n]\in\ov{A_d(X)}$. Now we let $t\rightarrow 0$, and obtain $[M_t'',n]\rightarrow [M(-q),n]$, whence $[M(-q),n]\in\ov{A_d(X)}$.
\item[(5)]$M\in\IM\rho$ with $h^0(X_n,M)\geq 2$ and $h^0(X_n,M(-p))=h^0(X_n,M(-q))=h^0(X_n,M)$. Then there exists $M'\in {\rm Pic}^{d'}X_n$, with $M=M'(hp+kq)$, $d'=d-(h+k)$, and $M'$ not supported on either $p$ or $q$ up to move the support away. As well as above, we consider $[M(-p),n]$ and $[M(-q),n]$ to show that they are contained in $\ov{A_d(X)}$. We proceed as in case (3) performing a double specialization, and recalling that $h^0(X_n,M')\geq 2$ by assumption.
\end{itemize}
Let $U\subset W_d(X_n)$ be the following set:
$$U:=\{M\in W_d(X_n)\mbox{ s.t. } h^0(X_n,M)=1,h^0(X_n,M(-p))=h^0(X_n,M(-q))=0\};$$
this is of course an open set in $W_d(X_n)$, and it contains all the line bundles $M$ studied in case (1). In particular for any $M\in U$, we have that $\ov{A_d(X)}$ intersects $\ov{F_M(X)}$ in only one point $L_M$, where $W_M(X)=\{L_M\}$. In order to verify this assertion, by (1) we just have to check that $[M(-p),n]$ and $[M(-q),n]$ are not contained in $\ov{A_d(X)}$, but this is obvious, since $h^0(X_n,M(-p))=0$, hence on the blow up $\hat X_n$ of $X$ at $n$, $h^0(\hat X_n,\widehat{M(-p)})=0$. From the study of all the possibilities above, from (2) to (5), we get that for any $M\in\IM\rho$ which is not in $U$, $\ov{A_d(X)}$ contains at least two points of $\ov{F_M(X)}$, but since the generic $M$ has $\sharp(\ov{F_M(X)}\cap\ov{A_d(X)})=1$, we have that for $M\in\IM\rho\setminus U$, the whole $\ov{F_M(X)}$ must be contained in $\ov{A_d(X)}$, hence for any $M\in\IM\rho$ we have that $\ov{W_M(X)}\subset\ov{A_d(X)}$.

So we have shown that $W_d(X)=A_d(X)$, with subsequent equality of their closures. In order to show that $\ov{W_d(X)}=\widetilde W_d(X)$, we argue like this: direction $\subset$ is obvious, since $\widetilde W_d(X)$ is a closed set in $\ov{P}_X^d$ containing $W_d(X)$. On the other hand, the analysis made above suggests that any $[N,n]\in\widetilde W_d(X)$ is also an element of $\ov{A_d(X)}$. Indeed if $N$ has $p$ and/or $q$ as base points, we argue as in (3),(4),(5); if otherwise $N$ does not contain $p$ nor $q$ in its support, by (2) we get that there exists $L(p_t)\in\IM\al_X^d$, such that $L(p_t)$ specializes to $[N,n]$ as $t\rightarrow 0$.

If the number of nodes $\D$ is $\geq 2$, we proceed by induction on $\D$. Indeed, let $X$ be a nodal irreducible curve having $\D$ nodes. We blow up $X$ at one node $n$, so that $\hat X_n$ is the blown up curve, and $X_n$ is the strict transform, and we have the normalization map $\nu:X_n\rightarrow X$ such that $\nu^{-1}(n)=\{p,q\}$. So again we look at the dominant morphism $\rho:W_d(X)\rightarrow W_d(X_n)$, and we prove that the fibers $W_M(X)\subset\ov{A_d(X)}$ for any $M\in\IM\rho$. As inductive hypothesis we assume that $W_d(X_n)=A_d(X_n)$ is irreducible of dimension $d$. This is the only point where we used the smoothness of $X_n$ in the previous case when $\D=1$; hence reapplying the argument above, which is based on \cite[Lemmas 2.2.3,2.2.4]{bib:theta}, we get the conclusions for every $\D$ and for every $d\geq 1$.
\end{proof}

\begin{remark}
{\rm We observe that when $d\geq g$, with $g$ the genus of $X$, it doesn't make sense referring to $W_d(X)$, since it is equal to $\Pic^dX$. On the other hand, when $d=1$ we have that by \cite[Lemma 2.2.3]{bib:theta}, $W_1(X)=\IM\al_X^1=A_1(X)$, and when $d=g-1$ we get that the Theta divisor is irreducible in $\Pic^{g-1}X$.}
\end{remark}

\begin{remark}\label{limit}
{\rm From the equality $A_d(X)=W_d(X)$ for any $d$, we deduce an important fact; we use the previous notation, where $X$ has $\D$ nodes and $X_n$ is the normalization at a node $n$. Let $L\in W_d(X)$ be such that $M=\nu^*L$ has $W_M(X)=F_M(X)$. Then $k^*=W_M(X)$, and we can denote its elements in the following way: $$W_M(X)=\{L^c\mbox{, }c\in k^*\}.$$ By \ref{Ad(X)} we have that for any $c\in k^*$ there exists a family $L_t^c\in\IM\al_X^d$ such that $L_t^c\rightarrow L^c$. In particular, we will have that $L_t^c=\tilde L_t^c(hp^c_t+kq^c_t)$ for suitable $h,k$, $p^c_t,q_t^c\in\dot X$ such that $\nu^*(p_t^c)$ specializes to $p$ on $X_n$, $\nu^*(q_t^c)$ specializes to $q$, and $\tilde L^c_t$ specializes to some effective line bundle on $X$ not supported on $n$. Hence we can assume $\tilde L_t^c=\tilde L^c$ not depending on $t$; so, for any $c\in k^*$, we have $\tilde L^c(hp^c_t+kq^c_t)\rightarrow L^c$. If $\tilde L^c$ is such that no other effective line bundle is in its fiber, we have that $\tilde L^c=\tilde L$, and $\tilde L(hp^c_t+kq^c_t)\rightarrow L^c$, so in this case the limit depends only upon the choice of the moving points $p_t^c$ and $q_t^c$. Equivalently, if $c\neq c'$ in $k^*$, there exist moving points $p_t^c,q_t^c$ and $p_t^{c'},q_t^{c'}$ such that $\tilde L(hp^c_t+kq^c_t)\rightarrow L^c$ and $\tilde L(hp^{c'}_t+kq^{c'}_t)\rightarrow L^{c'}$.}
\end{remark}

\section{Reducible curves}\label{redcurves}

Very little is known about Abel maps of reducible curves, even if recently a lot of effort has been put into studying the class of stable curves, see for example \cite{bib:theta}, \cite{bib:binary}, \cite{bib:linear}, \cite{bib:CoelTesi},\cite{bib:CoelPac}. We are going to study the relation among the varieties $W_\de(X)$, $A_\de(X)$ and their closures in $\Pbar$. Let $X$ be a reducible curve with components $C_1,\ldots,C_\ga$; for any $\de=(d_1,\ldots,d_\ga)\in\bZ^\ga$ with $|\de|=d$, we can consider the Brill-Noether variety $W_\de(X)$ that we defined in the introduction of the paper. Obviously if $d_i< 0$ for every $i=1,\ldots,\ga$, we get that $W_\de(X)=\emptyset$. On the other hand, if we assume $\de\geq\un 0$, i.e. $d_i\geq 0$ for every $i$, we can define the \textit{Abel map of multidegree} $\de$. Set $\dot X:=X\setminus X^{\rm sing}$, and $\dot C_i=C_i\cap\dot X$; we define $$\dot X^\de:=\dot C_1^{d_1}\times\ldots\times\dot C_\ga^{d_\ga}\subset X^\de:=C_1^{d_1}\times\ldots\times C_\ga^{d_\ga},$$
and
$$\begin{array}{rccc}
\al_X^\de:&\dot X^\de &\longrightarrow &\Pic^\de X \\
&(p_1,\ldots,p_d)&\mapsto &\Oc_X(\sum\limits_{i=1}^d p_i).
\end{array}$$
As in the irreducible case, we denote by $A_\de(X)$ the closure of the set ${\IM\al_X^\de}\subset\Pic^\de X$.
We are now going to introduce a set which will be crucial hereafter.
\begin{equation}\label{W^+}
W_\de^{+}(X):=\{L\in\Pic^\de X\mbox{ s.t. }h^0(Z,L|_Z)>0\mbox{ for any subcurve }Z\subseteq X\}.
\end{equation}

This definition suggests the following

\begin{lemma}\label{inclusion}
Let $\de\geq\un 0$ be a multidegree on a reducible curve $X$. Then $$A_\de(X)\subset W_\de^+(X).$$
\end{lemma}
\begin{proof}
The proof is straightforward: the line bundles in $\IM\al_X^\de$ are of the form $\Oc_X(\sum_{i=1}^d p_i)$, hence their restriction to any subcurve of $X$ has nonzero sections. Then by upper semicontinuity of the dimension of the $H^0$ this is still true for their limits in $A_\de(X)$.
\end{proof}

We start by studying the simplest case, i.e. when $X$ is a curve of compact type.

\subsection{{Curves of compact type}}
When $X$ is a curve of compact type, for any multidegree $\de$ we have that $\Pic^\de X$ is complete, hence so is $W_\de(X)$. However we are interested in the relation between $A_\de(X)$ and $W_\de(X)$. We start by assuming that $X$ has two smooth components $C_1,C_2$ meeting at one node $n$, hence its normalization is the disconnected curve $$C_1\sqcup C_2\stackrel{\nu}{\longrightarrow}X,$$ with $\nu^{-1}(n)=\{p,q\}$. This induces the pullback map
$$\Pic^{(d_1,d_2)} X\stackrel{\nu^*}{\longrightarrow}\Pic^{d_1} C_1\times\Pic^{d_2} C_2,$$
which is an isomorphism, and given $L\in W_\de(X)$, we denote $(L_1,L_2):=\nu^*(L)$. We define the sets:
\begin{equation}\label{components}
\begin{array}{rl}
W_\de^{+}(X):=&\{L\in W_\de(X)\mbox{ s.t. }h^0(C_1,L_1)> 0,h^0(C_2,L_2)>0\},\\&\\
W_\de^{+-}(X):=&\{L\in W_\de(X)\mbox{ s.t. }h^0(C_1,L_1)>0,h^0(C_2,L_2)=0\},\\ &\\
W_\de^{-+}(X):=&\{L\in W_\de(X)\mbox{ s.t. }h^0(C_1,L_1)=0,h^0(C_2,L_2)>0\};
\end{array}
\end{equation}
of course we have that $W_\de(X)=W_\de^{+}(X)\sqcup W_\de^{+-}(X)\sqcup W_\de^{-+}(X)$ set-theoretically.
\begin{prop}\label{compact2}
Let $X$ be a curve of compact type of genus $g$ with two smooth components $C_1,C_2$ of genus resp. $g_1,g_2$. Let $\de\geq\un 0$ be a multidegree with $|\de|=d$ such that $1\leq d\leq g-1$. We have:
\begin{itemize}
\item[(i)]if $d_1\leq g_1-1$ and $d_2\leq g_2-1$, then $W_\de(X)$ is connected and has $3$ irreducible components, of dimensions $d,d_1+g_2-1,d_2+g_1-1$,
\item[(ii)]if $d_1\geq g_1$ and $d_2\leq g_2-1$ (up to swapping the indices), $W_\de(X)$ is connected and has $2$ irreducible components.
\end{itemize}
\end{prop}
\begin{proof} In order to prove (i) we assume that $d_1\leq g_1-1$ and $d_2\leq g_2-1$. We consider the pullback map
$$\begin{array}{cccc}
\nu^*: & \Pic^\de X &\stackrel{\cong}{\longrightarrow}& \Pic^{d_1}C_1\times\Pic^{d_2}C_2\\
& L &\mapsto & (L_1,L_2);
\end{array}$$
then by \cite[2.1.1]{bib:theta} using that $\D=1$,
\begin{equation}\label{pullproduct}
W_\de^{+}(X)=(\nu^*)^{-1}(W_{d_1}(C_1)\times W_{d_2}(C_2)).
 \end{equation}
 Now since $C_1,C_2$ are smooth curves, we have that $W_{d_i}(C_i)$ is irreducible of dimension $d_i$ for $i=1,2$. Then $W_\de^{+}(X)$ is a closed irreducible set containing $A_\de(X)$. Since the fibers of $\nu^*$ have cardinality one, $\dim W_\de^{+}(X)=d$. By definition we know that $\IM\al_X^\de=(\nu^*)^{-1}(\IM\al^{d_1}_{C_1}\times \IM\al^{d_2}_{C_2})$, hence $\dim\IM\al_X^\de=d$, then $A_\de(X)=W_\de^{+}(X)$ and they both have dimension $d$.

The other two components of $W_\de(X)$ are the following ones: consider $L\in W_\de^{+-}(X)$; we have that $h^0(C_2,L_2)= 0$, and since $L$ has nonzero sections, we have $h^0(C_1,L_1(-p))> 0$. As in \cite{bib:survey} we define the set
\begin{equation}\label{lambda}
\La_p:=\{L_1\in\Pic^{d_1}C_1\mbox{ s.t. }h^0(C_1,L_1(-p))> 0\},
\end{equation}
and consider the isomorphism
\begin{equation}\label{iso}
\begin{array}{cccc}
\phi_p:&\Pic^{d_1-1}C_1&\longrightarrow&\Pic^{d_1}C_1\\
&M&\mapsto&M(p).
\end{array}
\end{equation}
It is easy to see that $\La_p=\phi_p(W_{d_1-1}(C_1))$, hence $\La_p$ is closed and irreducible of dimension $d_1-1$. Now consider the set $$\ov W_\de^{+-}(X):=(\nu^*)^{-1}(\La_p\times\Pic^{d_2}C_2);$$ it contains $$W_\de^{+-}(X)=(\nu^*)^{-1}(\La_p\times(\Pic^{d_2}C_2\setminus W_{d_2}C_2))$$ as an open set, and $\dim\ov W_\de^{+-}(X)=d_1+g_2-1$.

The last irreducible component of $W_\de(X)$ is the one containing the $L$'s  such that $h^0(C_1,L_1)=0$ and $h^0(C_2,L_2)\neq 0$. Arguing as before, we define the set $\La_q\subset\Pic^{d_2}C_2$, and the isomorphism $\phi_q:\Pic^{d_2-1}C_2\rightarrow\Pic^{d_2}C_2$ sending $N\in\Pic^{d_2-1}C_2$ to $N(q)$. Hence $\La_q=\phi_q(W_{d_2-1}(C_2))$, and the set $$\ov W_\de^{-+}(X):=(\nu^*)^{-1}(\Pic^{d_1}C_1\times\La_q)$$ is the closure of $W_\de^{-+}(X)$, with $\dim\ov W_\de^{-+}(X)=d_2+g_1-1$. Hence we have that
$$W_\de(X)=A_\de(X)\cup\ov W_\de^{+-}(X)\cup\ov W_\de^{-+}(X),$$
and their intersection is $(\nu^*)^{-1}(\La_p\times\La_q)$, having dimension $d_1-1+d_2-1=d-2$. This implies that $W_\de(X)$ is connected.

Part (ii) comes from part (i), once we have noticed that if $d_1\geq g_1$ and $d_2\leq g_2-1$, then $h^0(C_1,L_1)>0$, so $W_\de^{-+}(X)=\emptyset$. Hence $$W_\de(X)=W_\de^{+}(X)\cup\ov W_\de^{+-}(X),$$
and their intersection is $(\nu^*)^{-1}(\La_p)$, having dimension $d_1-1$. We notice that in this case by (\ref{pullproduct}), $\dim W_\de^{+}(X)=g_1+d_2$, which can be less than $d$. We prove that even in this case it holds that $W_\de^{+}(X)=A_\de(X)$. Indeed, inclusion ($\supset$) is obvious, and concerning ($\subset$), let us take a line bundle $L\in W_\de^{+}(X)$. Then we look at its pullback $M=\nu^*(L)$. Let $M=(\Oc_{C_1}(D_1+\la p),\Oc_{C_2}(D_2+\mu q))$ for some suitable divisors $D_1$ and $D_2$; we choose moving points $p_t$ on $C_1\cap X$ and $q_t$ on $C_2\cap X$, specializing resp. to $p$ and $q$. We consider on $C_1\sqcup C_2$ the line bundle:
$$M_t:=(\Oc_{C_1}(D_1+\la p_t),\Oc_{C_2}(D_2+\mu q_t)),$$
and push it down to $X$, getting the (unique) line bundle $L_t\in\IM\al_X^\de$ such that $\nu^*(L_t)=M_t$. Then if we let $t$ tend to $0$, we get that $L_t$ specializes to $L$, and hence that $L\in A_\de(X)$. So we conclude that $W_\de^{+}(X)=A_\de(X)$. It follows that $\dim W_\de(X)=\max\{g_1+d_2,d_1+g_2-1\}$. If vice-versa $d_2\geq g_2$ and $d_1\leq g_1-1$, we have that $W_\de^{+-}(X)=\emptyset$, $W_\de(X)=A_\de(X)\cup\ov W_\de^{-+}(X)$, $A_\de(X)\cap\ov W_\de^{-+}(X)=(\nu^*)^{-1}(\La_q)$,
and $\dim W_\de(X)=\max\{d_1+g_2,d_2+g_1-1\}$.
\end{proof}
\begin{remark}
{\rm We just observe that the case $d=g-1$ is carried out in \cite{bib:theta}, but we obtain it as a by-product in \ref{compact2}(ii); since there are no strictly balanced multidegrees summing to $g-1$ on a curve of compact type, we get that $W_\de(X)$ is not irreducible.}
\end{remark}

In the sequel we will try to generalize our study to any curve of compact type, so take $X$ as the union of irreducible smooth curves $C_1,\ldots,C_\ga$, with $g_i$ the genus of $C_i$ and $g$ the genus of $X$. Notice that since $X$ is of compact type, we have that $\sharp(C_i\cap C_j)=1$ for $i\neq j$, and this implies that the total number of nodes $\D\leq\ga-1$; we denote by $n_{ij}$ the intersection point $C_i\cap C_j$. Let $\de\geq\un 0$ be a multidegree on $X$, with $|\de|=d$, $1\leq d\leq g-1$. Let $$\nu:\bigsqcup\limits_{i=1}^\ga  C_i\rightarrow X$$ be the total normalization map, $\nu^*$ the pullback as before, and denote by $(L_1,\ldots,L_\ga)$ the pullback to $\bigsqcup\limits_{i=1}^\ga  C_i$ of any $L\in\Pic^\de X$. If $n_{ij}$ is a node, its branches on $C_i,C_j$ will be called respectively $p^i_j,p^j_{i}$, distinguishing the curve they belong to by the position of indices.
\begin{lemma}\label{genercomp}
Let $X$ be a connected curve of compact type as above and $\de\geq\un 0$. Then $W_\de^{+}(X)=A_\de(X)$, is a (closed) irreducible component of $W_\de(X)$.
\end{lemma}
\begin{proof}
The proof is straightforward: we see that, as we pointed out in the case $\ga=2$,
\begin{equation}\label{W1..1}
W_\de^{+}(X)=(\nu^*)^{-1}(W_{d_1}C_1\times\cdots\times W_{d_\ga} C_\ga),\end{equation}
indeed $X$ has a number of nodes $\D=\ga-1$, so we apply \cite[2.1.1]{bib:theta} and obtain the equality.
Since $C_i$ is smooth for every $i$, by (\ref{W1..1}) $W_\de^{+}(X)$ turns out to be a closed irreducible set of dimension $d_1+\cdots+d_\ga=d$, and it contains $A_\de(X)$. To see the inverse inclusion we argue as in \ref{compact2}(ii), proving that for any $L\in W_\de^{+}(X)$ there exists $L_t\in\IM\al_X^\de$ such that, if we let $t$ tend to $0$, we get that $L_t$ specializes to $L$. So we have $W_\de^{+}(X)=A_\de(X)$ as we wanted.
\end{proof}

What we are going to do now is to study the remaining irreducible components of $W_\de(X)$. To do this we need to introduce some notation: let $\D_i=\sharp(C_i\cap\ov{X\setminus C_i})$ for $i=1,\ldots,\ga$, and let $\un I$ be a $1\times\ga$ vector where the $j$-th component is $I_j=+$ or $I_j=-$. Then we can define the set:
$$W^{\un I}_\de(X):=\left\{L\in W_\de(X)\mbox{ s.t. }h^0(C_j,L_j)=0\mbox{ if }I_j=-\mbox{, and }h^0(C_j,L_j)>0\mbox{ if }I_j=+ \right\}.$$
Notice that if $I_j=+$ for every $j$, i.e. $\un I=(+,\ldots,+)$, we get $W_\de^{+}(X)$.
Let us fix some vector $\un I\neq (+,\ldots,+)$; set $$I^+:=\{j\in\{1,\ldots,\ga\},\,I_j=+\},$$ and $$I^-:=\{h\in\{1,\ldots,\ga\},\,I_h=-\}.$$
We denote by $p^{j}_h$ the branch on $C_{j}$ of the point of intersection $C_{j}\cap C_h$, for $j\in I^+$, and some $h\in I^-$, if it exists. Moreover, we fix $j\in I^+$ and consider the disconnected curve $\ov{X\setminus C_j}=X^j_1\sqcup\cdots\sqcup X^j_{k_j}$. We observe that $C_j$ has only one point of intersection with each $X^j_l$, for $l=1,\ldots,k_j$. We denote the branches of this point on $C_j$ and $X^l_j$ resp. by $p^j_l$ and $p^l_j$. If $(L_1,\ldots,L_\ga)$ are the restrictions of a line bundle $L$ on $X$ to each irreducible component of $X$, we denote by $L_{X^l_j}$ the restriction of $L$ to the connected component $X^l_j$.
Set:
$$\Li_j:=\left\{l\in\{1,\ldots,k_j\},\phantom{\dagger}h^0(L_{X^l_j})=0\right\},$$
and let

$$\La_{j}:=\{L_j\in W_{d_j}(C_j),\,h^0(L_j(-\sum_{l\in\Li_j}p^j_l))> 0\}.$$

Now, still for $j\in I^+$, consider the set:
\begin{equation}\label{formula}
\widetilde\Sigma_j:=\left(\prod_{h\in I^-}(\Pic^{d_h}C_h\setminus W_{d_h}(C_h))\times\La_{j}\times\prod_{l\in I^+,l\neq j}W_{d_l}(C_{l})\right)\subset\prod_{i=1}^\ga\Pic^{d_i}C_i.
\end{equation}
and denote by $\Sigma_j$ the set obtained from $\widetilde\Sigma_j$ by reordering the factors in such a way that the final order in $\Sigma_j$ corresponds to the order of the components of $\un I$, so for example, $\La_j$ will be the factor in the position of $j$ in $\un I$.
We will denote by
\begin{equation}\label{totalset}
\Sigma_{\un I}=\bigcup\limits_{j\in I^+}\Sigma_j.
\end{equation}
We observe that $\La_j$ is irreducible (see \ref{compact2}), and its dimension depends on the cardinality $\la_j$ of $\Li_j$. Indeed $\dim(\La_j)=d_j-\la_j$ and $0\leq\la_j\leq k_j$. It follows that $\Sigma_j$ is irreducible for every $j\in I^+$.
\begin{lemma}
We have that $(\nu^*)^{-1}(\Sigma_{\un I})=W^{\un I}_\de(X)$.
\end{lemma}
\begin{proof}
Inclusion ($\subset$) is easy by definition of $\Sigma_{\un I}$, since an element of $(\nu^*)^{-1}(\Sigma_{\un I})$ must have at least a nonzero section. On the other hand, given a line bundle $L\in W^{\un I}_\de(X)$, we want to prove that $\nu^*(L)=(L_1,\ldots,L_\ga)$ belongs to $\Sigma_{\un I}$. If for every $i\in I^+$, $h^0(L_{X^l_i})\neq 0$ for every $l=1,\ldots,k_i$, then $\Li_i=\emptyset$ and $\La_i=W_{d_i}(C_i)$ for every $i$, hence in this case $$\Sigma=\prod_{h\in I^-}(\Pic^{d_h}C_h\setminus W_{d_h}(C_h))\times\prod_{l\in I^+}W_{d_l}(C_{l})$$ up to reordering the factors in the left hand side, and therefore $\nu^*(L)\in\Sigma$. Now, assume that there exists $i\in I^+$ such that $\Li_i\neq\emptyset$. Without loss of generality we can assume that $|\Li_i|=1$. Then in order to glue the sections and get a line bundle on $X$, it must be $h^0(L_i(-p^i_l))> 0$, hence $L_i\in\La_i$, and therefore $$(L_1,\ldots,L_\ga)\in\Sigma_j\subset\Sigma_{\un I}.$$
\end{proof}

Even if we can't say precisely which is the dimension of the components of $W^{\un I}_\de(X)$, we can count how many they are. By (\ref{totalset}) we see that for any fixed $\un I$, the number of irreducible components of $W^{\un I}_\de(X)$ is $|I^+|$. Hence we can say that the number of irreducible components of $W_\de(X)$ is \begin{equation}\label{numcomp}
\mathcal N:=\left(1+\sum_{\un I\neq(+,\ldots,+)}|I^+|\right).
\end{equation}
\begin{remark}
{\rm We notice that depending on $\de$, some $\un I$'s won't appear in (\ref{numcomp}); indeed, if there exists some $k\in\{1,\ldots,\ga\}$ such that $d_k\geq g_k$, then the component $I_k$ of $\un I$ must be $+$, so we will have a small number of $\un I$'s, and hence a small number of irreducible components in $W_\de(X)$. Moreover, if $d_k\geq g_k$ for every $k$, we get that the only irreducible component of $W_\de(X)$ is $W_\de^{+}(X)$.}
\end{remark}

\subsection{Binary curves}
A binary curve of genus $g$ is a nodal curve made of two smooth rational components intersecting at $g+1$ points. We are going to recall some properties that we will use throughout this paragraph. If $X$ is a binary curve of genus $g\geq-1$, a multidegree $\de=(d_1,d_2)$ such that $|\de|=d$, is balanced on $X$ if
\begin{equation}
m(d,g):=\dfrac{d-g-1}2\leq d_i\leq \dfrac{d+g+1}2=:M(d,g)
\end{equation}

We say that $\de$ is strictly balanced if strict inequality holds. If $\wh X_S$ is a quasistable curve obtained from a binary curve $X$ by blowing up the nodes in $S$, then we call $E_1,\ldots,E_{\sharp S}$ the exceptional components, so that if $X_S$ is the partial normalization of $X$ at the nodes in $S$, we have that $\wh X_S=X_S\cup \cup_{i=1}^{\sharp S}E_i$.

\begin{defi}
{\rm A multidegree $\wh\de=(d_1,\ldots,d_{2+\sharp S})$ on $\wh X_S$ with $|\wh\de|=d$, is balanced if the following hold:\\
(1) $d_i=1$ for any $i=3,\ldots,\sharp S$, i.e. $\wh\de |_{E_i}=1,\forall i$.\\
(2) $\wh\de|_{X_S}$ is balanced on $X_S$.\\
$\wh\de$ is strictly balanced if its restriction to $X_S$ is strictly balanced on $X_S$.}
\end{defi}

\begin{remark}\label{balanceNorm}
{\rm Let $X$ be a binary curve of genus $g$, and let $X_n$ be the normalization of $X$ at the node $n$,
such that $\nu:X_n\rightarrow X$ is the associated map. Let $\de\geq\un 0$ be a balanced multidegree on $X$ such that $|\de|=d\leq g-1$, then it is still balanced on $X_n$. Indeed, let us suppose by contradiction that
$$d_1<m(d,g-1);$$ then it should be
$$d_1<\dfrac{d-g}2\leq\dfrac{g-1-g}2,$$ but then we would have that $d_1<0$, which cannot happen.}
\end{remark}

\begin{lemma}\label{Lucia}
Let $X$ be a quasistable curve, and $L\in\Pic^d X$ a balanced line bundle such that $\un{\deg}L=:\de$ with $d\leq g-1$ and $h^0(L)\geq 1$.
Then there exists a non exceptional irreducible component $C$ of $X$ such that for general $p\in C$
$$h^0(L(p))=h^0(L).$$
\end{lemma}
\begin{proof}
We fix a smooth point $p$ on $X$. We know that $h^0(L(p))\geq h^0(L)$. We suppose that $h^0(L(p))= h^0(L)+1$; by Riemann-Roch this is equivalent to saying that $h^0(\omega_X\otimes L^{-1}(-p))=h^0(\omega_X\otimes L^{-1})$. This holds if and only if $p$ is a base point of $\omega_X L^{-1}$. But now we notice that, again by Riemann-Roch theorem, $$h^0(\omega_X\otimes L^{-1})=h^0(L)+2g-2-d-g+1\geq h^0(L)\geq 1.$$ Therefore $\omega_X\otimes L^{-1}$ has some non vanishing section on $X$. If $E\subset X$ is an exceptional component, then $\deg_E\omega_X=0$ and $\deg_E L=1$, hence $\deg_E\omega_X\otimes L^{-1}=-1$, hence every section of $\omega_X\otimes L^{-1}$ vanishes on $E$. This implies that there must be a non exceptional component $C$ of $X$ such that the restriction to $C$ of $H^0(\omega_X\otimes L^{-1})$ is non zero.
Hence the general point $p\in C$ is not a base point of $\omega_X\otimes L^{-1}$.
So we get our conclusions.
\end{proof}

\begin{remark}\label{binary(g-1)}
{\rm We recall that if $X$ is a nodal curve and $\de=\un{g-1}$ is \emph{stably balanced} as in section 1.3.1 in \cite{bib:theta}, then $W_\de(X)=A_\de(X)$. It's very easy to see that if $X$ is a binary curve and $\de=\un{g-1}\geq\un 0$ is balanced, then $\de$ is strictly balanced and hence stably balanced. This implies that if $X$ is a binary curve of genus $g$ and $\de=\un{g-1}\geq\un 0$ balanced, then $W_\de(X)=A_\de(X)$.}
\end{remark}

\begin{lemma}\label{balanced+1}
Let $\de\in\ov{B_d(X)}$ be such that $W_\de(X)\neq\emptyset$ where $X$ is binary of genus $g$ and $d\leq g-1$. Then $\de\geq\un 0$ and $\de\in B_d(X)$.
\end{lemma}
\begin{proof}
By \cite{bib:binary}[Proposition 12] if $d_i<0$ and $d\leq g$ we have $W_\de(X)=\emptyset$, hence $\de\geq\un 0$. Now we have
$$m(d,g)=\dfrac{d-g-1}2\leq\dfrac{g-1-g-1}2=-1. $$ Therefore, if $\de\geq\un 0$, $d_i\neq m(d,g)$ for $i=1,2$. Hence $\de$ is strictly balanced on $X$.
\end{proof}

We notice that by lemma \ref{balanced+1}, for a binary curve we have $W_\de^+(X)=W_\de(X)$.

\begin{prop}\label{closure}
Let $X=C_1\cup C_2$ be a binary curve of genus $g$, $L$ a line bundle on $X$ of degree $\de$ balanced, with $0<|\de|\leq g-1$, and $h^0(X,L)>0$. Then there exists a family $L_t\in\mathrm{Im}\alpha_X^\de$ such that $L_t\rightarrow L$ when $t\rightarrow 0$.
\end{prop}

\begin{proof}
Let $L$ be a line bundle as in the hypothesis; we will use induction on the degree.

If $d=g-1$ by \cite{bib:theta} (see remark \ref{binary(g-1)}) we have that $A_\de(X)=W_\de(X)$.

Now let $d<g-1$; by lemma \ref{Lucia} we have that there exists a component of $X$, say $C_1$, such that for the general $p\in C_1$ we have that $h^0(L(p))=h^0(L)$. By lemma \ref{balanced+1} $L(p)$ has balanced multidegree on $X$. Hence we can apply induction and get that there exists a family $L'_t\in\mathrm{Im}\alpha_X^{\un{d+1}}$ such that $L'_t\rightarrow L(p)$. Like before we denote this family via
\begin{equation}\label{family2}
\mathcal O_X(a^1_t+\cdots+a^{d+1}_t)\rightarrow L(p).
\end{equation}
We notice that $p$ is a base point of $L(p)$. Let $\nu:X_n\rightarrow X$ be the normalization of $X$ at a node $n$, as in remark \ref{balanceNorm}. Then we can pullback (\ref{family2}) to $X_n$ and get
\begin{equation}\label{family3}
\mathcal O_{X_n}(a^1_t+\cdots+a^{d+1}_t)\rightarrow L'(p),
\end{equation}
where with abuse of notation we call the points on $X$ and $X_n$ in the same way, and $L'=\nu^*(L)$.

Now we divide the proof in two cases:
\begin{itemize}
\item[{\bf Case 1}]: we assume that $h^0(L'(p))=h^0(L')$.

We need to use a second induction on the number of nodes. The inductive statement is: if $\wt L$ and $\wt L(p)$ are balanced line bundles on $Y$ binary curve with $\delta $ nodes, with $\un\deg\wt L=\de\geq\un 0$ with $h^0(\wt L)>0$, $h^0(\wt L(p))=h^0(\wt L)$ and there exists $O_{Y}(a^1_t+\cdots+a^{d+1}_t)\rightarrow \wt L(p)$, then $a^i_t\rightarrow p$ for some $i$.

The base of induction is obvious on a curve with no nodes, i.e. a smooth one. So we suppose that the statement above is true for $X_n$: in particular we know that $h^0(L'(p))=h^0(L')$; then by induction it holds that in (\ref{family3}) there exists $a^i_t$ such that
\begin{equation}\label{points}
a^i_t\rightarrow p \mbox{ for some }i.
\end{equation}
Up to reordering the points we can assume that $i=d+1$. Now, by applying (\ref{points}) to (\ref{family2}) we get that $$\mathcal O_X(a^1_t+\cdots+a^{d}_t)\rightarrow L,$$
and hence the conclusions in case 1.

\item[{\bf Case 2}]: we assume that $h^0(L'(p))=h^0(L')+1$. Then we have $h^0(L')=h^0(L)$. We have two possibilities: by applying Lemma 2.2.3 (2) and Lemma 2.2.4 (2) in \cite{bib:theta}, either $n$ is a base point of $L$, or $W_{L'}(X)=\{L\}$. In the first case we have that it must be true regardless of the choice of $n$, i.e. every node $n$ of $X$ must be a base point of $L$, which is impossible since the nodes are $g+1$ whereas the degree of $L$ is $d<g-1$.

    On the other hand, if $W_{L'}(X)=\{L\}$ we need a new inductive argument on the number of nodes. In this case the inductive statement is: let $Y$ is a binary curve of genus $g$, $M\in\Pic^\de Y$ such that $\de$ is balanced and $d\leq g-1$ with $h^0(Y,M)>0$. Then there exists $M_t\in\mathrm{Im}\alpha_Y^\de$ such that $M_t\rightarrow M$ when $t\rightarrow 0$.

The base of induction is given by a binary curve of genus $2$, i.e. with $3$ nodes, so that $d=1$, and since $d=g-1$, by \cite{bib:theta} we have $W_\de(X)=A_\de(X)$, hence the conclusion holds.

We assume the inductive statement for $X_n$, so we get that there exists $L'_t\in\mathrm{Im}\alpha_{X_n}^\de$ such that $L'_t\rightarrow L'$. Since $L'_t\in\mathrm{Im}\alpha_{X_n}^\de$, for every $t$ there exists $L_t\in\IM\alpha_X^\de$ such that $\nu^*(L_t)=L'_t$. By the fact that $W_{L'}(X)=\{L\}$, we conclude that $L_t\rightarrow L$.
\end{itemize}
\end{proof}

\begin{cor}
Let $X$ be a binary curve of genus $g$, and let $\de\geq\un 0$ be a balanced multidegree on $X$. Then $W_\de(X)=A_\de(X)$. In particular $W_\de(X)$ is irreducible of dimension $d$.
\end{cor}

\begin{proof}
The first assertion is implied by proposition \ref{closure}. And of course this implies that $W_\de(X)$ is irreducible. By \cite{bib:binary}[proposition 25] we have that the dimension of $W_\de(X)=d$.
\end{proof}

So far we have studied the closure of $\mathrm{Im}\alpha_X^\de$ inside $\Pic^\de X$ when $X$ is a binary curve. The next step is to study its closure inside the compactified Picard variety $\Pbar$.

Let $B_d(X)$ be the set of strictly balanced line bundles of multidegree $\de$ on $X$, with $|\de|=d$, and denote by $\ov{B_d(X)}$ the set of balanced multidegrees. The stratification of $\Pbar$ as in \cite[Fact 2.2]{bib:Néron} is the following
\begin{equation}\label{strata}
\Pbar=\coprod_{\emptyset\subset S\subset X^{\rm sing} \atop \de\in B_{d}(\widehat X_S)}P_S^{\de}.
\end{equation}

For any set $S$ of nodes of $X$, if $C_1,C_2$ are the smooth components of $X$, $X_S=C_1\cup C_2$, with $\D_S=\sharp(C_1\cap C_2)=\D-\sharp S$, so that the total normalization is

$$C_1\sqcup C_2\stackrel{\nu_S}{\longrightarrow}X_S,$$
and given $L\in W_{\de_S}(X_S)$, we denote $(L_1,L_2):=\nu_S^*(L)$. The stratification in (\ref{strata}) motivates the definition of
\begin{equation}\label{W+}
\wt W_d(X):=\bigsqcup\limits_{\emptyset\subset S\subset X^{\rm sing} \atop \wh\de_S\in B^{^{\geq0}}_{d}(\wh X_S)}W_{\de_S}(X_S),
\end{equation}
where $B^{^{\geq0}}_{d}(\wh X_S)$ is the set of strictly balanced multidegrees $\un e\geq\un 0$ on $\wh X_S$ such that $|\un e|=d$, $\wh X_S$ is the partial blow up of $X$ at the nodes contained in $S$, $X_S$ is the strict transform of $X$, $\wh \de_S$ is a balanced multidegree on
$\wh X_S$ such that $|\wh \de_S|=d$, whereas $$\de_S={\widehat \de}_S{|_{X_S}}$$ and $|\de_S|=d-\sharp S$.

We notice that, if $\de\in B^{^{\geq0}}_d(X)$, denoting by $\ov{A_\de(X)}$ the closure of $A_\de(X)$ in $\Pbar$, similarly to lemma \ref{inclusion} we have the inclusion
\begin{equation}\label{inclusion2}
\ov{A_\de(X)}\subset \wt W_d(X).
\end{equation}

\begin{defi}
{\rm We denote by} $$\ov{A_d(X)}:=\ov{\bigcup_{\de\in B^{^{\geq0}}_d(X)}A_\de(X)}\subset \Pbar.$$
\end{defi}

\begin{teo}\label{totalclosure}
Let $X=C_1\cup C_2$ be a binary curve of genus $g\geq 2$ with $\D\geq 2$ nodes and smooth components. Take $1\leq d\leq g-1$. Then
$$\wt W_d(X)=\ov{A_d(X)}\subset \Pbar.$$
\end{teo}
\begin{proof}
Let us observe that, since $\sharp B^{^{\geq0}}_d(X)<\infty$, we have that
$$\ov{\bigcup_{\de\in B^{^{\geq0}}_d(X)}A_\de(X)}=\bigcup_{\de\in B^{^{\geq0}}_d(X)}\ov{A_\de(X)}.$$
For any $\de\in B^{^{\geq0}}_{d}(X)$, by (\ref{inclusion2}) we get that inclusion ($\supset$) holds. Let us now prove inclusion ($\subset$). By (\ref{W+}) it is sufficient to show that for any $\emptyset\subset S\subset X^{\rm sing}$, $$W_{\de_S}(X_S)\subset\ov{A_\de(X)}$$ for a certain $\de\in B^{^{\geq0}}_d(X)$.

First of all we notice that by (\ref{W+}), we can equivalently write
$$\wt W_d(X)=\left\{[M,S]\in\Pbar\mbox{ s.t. }M\in W_{\de_S}(X_S)\mbox{ with }\de_S\geq\un 0\right\}.$$
Let us assume that $\sharp S=1$, with $S=\{n\}$; take $M\in W_{\de_S}(X_S)$ with $\de_S={\wh\de_S}|_{X_S}$ and $\wh\de_S\in\ov{B^{^{\geq0}}_{d}(\wh X_S)}$, and consider $[M,S]=[M,n]$. Thanks to the stratification of $\Pbar$ there exists $\de=(d_1,d_2)\in B^{^{\geq0}}_d(X)$ such that either $\de_S=(d_1-1,d_2)$ or $\de_S=(d_1,d_2-1)$. We assume, with no loss of generality, that $\de_S=(d_1-1,d_2)$. Now by proposition \ref{closure} we know that $M\in A_{\de_S}(X_S)$, i.e. there exists a family $M_t\in\IM\al_{X_S}^{\de_S}$ such that $M_t$ specializes to $M$ on $X_S$ as $t\mapsto 0$. By \cite[Lemmas 2.2.3,2.2.4]{bib:theta} we have that for any $t$ there exists $L_t\in\IM\al_X^{\de_S}$ such that the pullback of $L_t$ to $X_S$ is $M_t$. Let us fix $t$ and take a moving point $p_u$ on $C_1\cap\dot X$ such that $p_u$ specializes to $n$ as $u\mapsto 0$. We see that by construction $\grad L_t(p_u)=\de$ and $L_t(p_u)\in\IM\al_X^{\de}$; moreover
$$L_t(p_u)\stackrel{u\rightarrow 0}{\longrightarrow}[M_t,n]\in\ov{A_\de(X)}.$$
Now we let $t$ tend to $0$, so we obtain that $[M_t,n]\mapsto[M,n]$, and $[M,n]\in\ov{A_\de(X)}$.

We proceed by induction on $\sharp S$; we have just proved that when $\sharp S=1$, then for any $\wh\de_S\in\ov{B^{^{\geq0}}_{d}(\wh X_S)}$ there exists $\de\in B^{^{\geq0}}_d(X)$ such that $W_{\de_S}(X_S)\subset\ov{A_\de(X)}$. Let us suppose that $S\subset X^{\rm sing}$ is such that for any $\wh\de_S\in\ov{B^{^{\geq0}}_{d}(\wh X_S)}$ there exists $\de\in B^{^{\geq0}}_d(X)$ such that $W_{\de_S}(X_S)\subset\ov{A_\de(X)}$, or equivalently, $[M,S]\in\ov{A_\de(X)}$ for every $M\in W_{\de_S}(X_S)$.

We want to prove that for $T\subset X^{\rm sing}$, with $T=S\cup n$ for any node $n$ of $X_S$, then taking  $\wh\de_T\in\ov{B^{^{\geq0}}_{d}(\wh X_T)}$ there exists $\de\in B^{^{\geq0}}_d(X)$ such that $[M_T,T]\in\ov{A_\de(X)}$ with $M_T\in W_{\de_T}(X_T)$.

We take an element $M_T\in W_{\de_T}(X_T)$, and consider $[M_T,T]$. By \ref{closure} we know that $M_T\in A_{\de_T}(X_T)$, hence there exists a family $M_T^t\in\IM\al_{X_T}^{\de_T}$ such that $M_T^t$ specializes to $M_T$ on $X_T$ as $t\mapsto 0$. Let $\de_S=(d_1^S,d_2^S)$ be a multidegree on $X_S$ such that $|\de_T|=|\de_S|-1$ and $\wh\de_S\in\ov{B^{^{\geq0}}_{d}(\wh X_S)}$; it exists because of the stratification of $\Pbar$. Let us assume that, say, $\de_T=(d_1^S-1,d_2^S)$. Again \cite[Lemmas 2.2.3,2.2.4]{bib:theta} imply that for any $t$ there exists $M_S^t\in\IM\al_{X_S}^{\de_T}$ such that the pullback of $M_S^t$ to $X_T$ is $M_T^t$. We fix $t$ and take $p_u\in C_1\cap\dot X$ specializing to $n$ on $X_S$; then we have that $\grad M^t_S(p_u)=\de_S$, hence by inductive hypothesis $[M^t_S(p_u),S]\in\ov{A_\de(X)}$ for a certain $\de\in B^{^{\geq0}}_d(X)$. Then we have that $$[M^t_S(p_u),S]\stackrel{u\rightarrow 0}{\longrightarrow}[M^t_T,T]\in\ov{A_\de(X)},$$
and again letting $t$ tend to $0$ we obtain that $$[M_T^t,T]\stackrel{t\rightarrow 0}{\longrightarrow}[M_T,T]\in\ov{A_\de(X)},$$
as we wanted.

It follows that
$$\bigsqcup\limits_{\emptyset\subset S\subset X^{\rm sing} \atop \wh\de_S\in B^{^{\geq0}}_{d}(\wh X_S)}W_{\de_S}(X_S)=\bigcup_{\de\in B^{^{\geq0}}_d(X)}\ov{A_\de(X)},$$
hence we get the conclusions.
\end{proof}

We are now going to investigate about the closure inside $\Pbar$ of the set $A_\de(X)$ when $\de$ is a strictly balanced multidegree on $X$ binary curve. Before, we need to recall some definitions introduced in \cite{bib:compJac}.

\begin{defi}\label{dominant}
{\rm Let $X$ and $\wh X$ be two Deligne-Mumford semistable curves; we say that $\wh X$ \textit{dominates} $X$ if they have the same stable model and if there exists a surjective morphism of $\wh X$ onto $X$ such that every component of $\wh X$ is either contracted to a point or mapped birationally onto its image.}
\end{defi}

\begin{defi}\label{refinement}
{\rm Let $\de\in\ov{B_d(X)}$ and $\wh\de\in\ov{B_d(\wh X)}$. We say that $\wh\de$ is a \textit{refinement} of $\de$, and we denote it by $$\wh\de\preceq\de,$$ if and only if $\wh X$ dominates $X$ via a map $\varphi$ and for every subcurve $Y$ of $X$ there exists a subcurve $\wh Y$ of $\wh X$ such that $\varphi$ maps $\wh Y$ to $Y$ and $|\de_Y|=|\wh\de_{\wh Y}|$}.
\end{defi}

We are now able to state:

\begin{prop}
Let $X$ be a binary curve of genus $g\geq 2$ and $\de\geq\un 0$ a strictly balanced multidegree on $X$. Then $$\ov{A_\de(X)}=\bigsqcup\limits_{{\emptyset\subset S\subset X^{\rm sing} \atop \wh\de_S\in B^{^{\geq0}}_{d}(\wh X_S) }\atop \wh\de_S\preceq\de}W_{\de_S}(X_S)\subset \wt W_d(X).$$
\end{prop}

\begin{proof}
Inclusion $(\subset)$ is obvious by an argument analogous to lemma \ref{inclusion}. The proof of $(\supset)$ is actually the same as in \ref{totalclosure}, i.e. we take an element $[M,S]\in\Pbar$ such that $M\in W_{\de_S}(X_S)$ and $\wh\de_S\preceq\de$ and we use the same argument as in \ref{totalclosure}, considering that $\wh\de_S\preceq\de$, hence $W_{\de_S}(X_S)\subset\ov{A_\de(X)}$.
\end{proof}

\subsubsection{\bf Degree $1$}\label{degree1}

We are now going to investigate what happens when the degree $d=1$. Let $X$ be a nodal connected curve of genus $g\geq 2$, let $C_1,\ldots,C_\ga$ be its irreducible components, and set $g_i=g(C_i)$, and $\D_i=\sharp(C_i\cap\ov{X\setminus C_i})$.

\begin{lemma}\label{negdeg1}
Let $X$ be a semistable curve of genus $g\geq 2$ as above and $\de$ be a balanced multidegree on $X$ such that $|\de|=1$ and $\de\ngeqslant \un 0$. Then $W_\de(X)=\emptyset$.
\end{lemma}

\begin{proof}
Let us suppose that $C_i$ is an irreducible component of $X$ such that $d_i<0$. By the balancing condition we know that:
\begin{equation}\label{balance}
\dfrac{2g_i-2+\D_i}{2g-2}-\dfrac{\D_i}2\leq d_i\leq \dfrac{2g_i-2+\D_i}{2g-2}+\dfrac{\D_i}2.
\end{equation}
 Assume that there exists $L\in W_\de(X)$, and denote by $n_1,\ldots,n_{\D_i}$ the nodes of $C_i\cap\ov{X\setminus C_i}$. Moreover, denote for simplicity $Z_i=\ov{X\setminus C_i}$ and let $q_1,\ldots,q_{\D_i}$ be the branches of $n_1,\ldots,n_{\D_i}$ on $Z_i$. If $Y\subset X$ is a subcurve of $X$, we denote by $L_Y:=L_{|_{Y}}$.  Since by assumption $h^0(C_i,L_{C_i})=0$, we have that
$$h^0(X,L)=h^0(Z_i,L_{Z_i}(-q_1-\cdots-q_{\D_i}))> 0.$$
Hence we must have that $\deg L_{Z_i}\geq \D_i$, and recalling that $\deg L_{C_i}=d_i=1-\deg L_{Z_i}$, it follows that
\begin{equation}\label{degdiseq}
d_i\leq 1-\D_i.
\end{equation}
Therefore we have to verify that
$$\dfrac{2g_i-2+\D_i}{2g-2}-\dfrac{\D_i}2\leq1-\D_i;$$
this holds if and only if $$\D_i\leq 2-\dfrac{2g_i}g.$$
So we have two possibilities:
\begin{itemize}
\item[(i)]either $g_i=0$ and $\D_i=1$ or $\D_i=2$,
\item[(ii)]or $g_i\neq 0$ and $\D_i=1$.
\end{itemize}
In case (i), $C_i\cong\bP^1$, so let us suppose that $\D_i=2$; hence $C_i$ is an exceptional component of $X$, and by the balancing condition it must be $d_i=1$. But by (\ref{degdiseq}), we see that $d_i\leq -1$, and we get a contradiction.

Suppose now that $g_i\geq 0$ and $\D_i=1$. Then, by (\ref{balance}) we have that
$$d_i\geq\dfrac{-1}2+\dfrac{2g_i-1}{2g-2},$$
hence $d_i\geq 0$, which is again a contradiction in both cases (i) and (ii). Therefore $W_\de(X)=\emptyset$.
\end{proof}

By lemma \ref{negdeg1} we have that the only possibility that $W_\de(X)\neq\emptyset$ is when $\de=(1,0,\ldots,0)$, up to swapping some indices. In particular, when $\de=(1,0,\ldots,0)$ and $L\in W_\de(X)$ we have that by \cite[Lemma 4.2.3]{bib:linear} either $h^0(L)\leq 1$ or $C_1$ is a separating line, $h^0(L)=2$ and $L_{\ov{X\setminus C_1}}=\Oc_{\ov{X\setminus C_1}}$. We have the following

\begin{teo}\label{Ad(X)deg1}
Let $X$ be a connected nodal curve; let $\de=(1,0,\ldots,0)$ be a multidegree on $X$. Then $A_\de(X)=W^+_\de(X)$.
\end{teo}

\begin{proof}
By lemma \ref{inclusion} we only need to prove that $A_\de(X)\supset W^+_\de(X)$.
As usual we call $C_1,\ldots,C_\gamma$ the irreducible components of $X$, then up to reordering we have that $\de|_{C_1}=1$. By \cite{bib:linear}[Lemma 4.2.3] if $L\in W^+_\de(X)$ we must have $h^0(L)=1$ unless $C_1$ is a separating line of $X$, we will discuss this case later, so suppose $C_1$ is not a separating line of $X$.

Then $L$ has one nonzero section $s$ on $X$. If it vanishes on a smooth point $r$ of $X$, we have that $L=\Oc_X(r)$, so $L\in\IM\al_X^\de$. Otherwise, if there exists a node $n\in C_1\cap C_1^c$ such that $s(n)=0$, we normalize $X$ at $n$; we denote $X'\stackrel{\nu}{\rightarrow}X$ the normalization at $n$. Now we have two possibilities:

(i) $X'$ is connected, i.e. $n$ is nonseparating. Let $L'$ be the pullback of $L$ to $X'$. Let us denote by $C_2$ the component of $X$ such that $\{n\}=C_1\cap C_2$ and by $p,q$ the branches of $n$ on $X'$ with $p\in C_1$. With abuse of notation we call again $s$ the pullback of $s$ to $L'$. Then $s(p)=s(q)=0$, but $L'$ has degree one, so $s$ doesn't vanish on other points of $C_1$, and in particular if $\{n_1,\ldots,n_l\}$ are the other nodes of $\ov{X'\setminus C_1}\cap C_1$, $s(n_i)\neq 0$. Notice that $l\geq 1$. Since the pullback of $L$ to the total normalization of $X$ is $(\Oc_{C_1}(p),\Oc_{C_2},\ldots,\Oc_{C_\gamma})$, then $s$ restricted to $X'\setminus C_1$ must be a constant, hence by what we just said a nonzero constant. In particular $s(q)\neq 0$, which is a contradiction. Therefore $s$ cannot vanish on a nonseparating node of $X$.

(ii) $X'$ is not connected. Then $X'=C_1\sqcup Z_1$ and $Z_1$
 is connected. In particular $n$ is a separating node of $X$. The pullback of $L$ to $X'$ is $M=(\Oc_{C_1}(p),\Oc_{Z_1})$, where $p$ is as in (i). Now let us consider a moving point $p_t\in C_1\setminus p$, such that $p_t\mapsto p$. Let $M_t=(\Oc_{C_1}(p_t),\Oc_{Z_1})\in\Pic X'$. Then the line bundle $\Oc_X(p_t)$ on $X$ pulls back to $M_t$ (abusing notation). As $p_t\mapsto p$, $M_t\mapsto M$ and $L_t\mapsto\wt L$ such that $h^0(\wt L)\geq 1$. Since $h^0(X',M)=1$ we have that $L$ is the unique line bundle on $X$ pulling back to $M$ and such that $h^0(L)\geq 1$. Hence $L=\wt L$. Hence $L\in A_\de(X)$.

If $C_1$ is a separating line, since $L\in W^+_\de(X)$ we have $h^0(L)=2$. Then we can choose $r\in C_1\setminus (C_1\cap C_1^c)$ such that $L=\Oc_X(r)$. So $L\in\IM\al_X^\de$.
\end{proof}

In what follows we are going to give a characterization of the closure of $A_\de(X)$ in $\Pbar$ for stable curves in degree $1$. Let then $X$ be a stable curve. Let as usual $X=C_1\cup\cdots\cup C_\ga$ be the decomposition of $X$ into irreducible components. If $X_S$ is a partial normalization of $X$ at a set $S$ of nodes, we consider the decomposition of $X_S$ in connected components:
$$X_S=X_1^S\sqcup\cdots\sqcup X_{r_S}^S.$$
We denote by
$$\wt\nu_S:\bigsqcup_{i=1}^\ga C_i^S\longrightarrow X_S,$$
the partial normalization of $X_S$ at all the nodes in the set $$\bigcup_{i=1}^\ga \left(C_i^S\cap(X_S\setminus C_i^S)\right).$$
We recall that by \cite{bib:compJac}, for any stable curve of genus $g\geq 2$ and any $d$, we have a decomposition
$$\Pbar=\coprod_{\emptyset\subset S\subset X^{\rm sing} \atop \de\in B_{d}(\widehat X_S)}P_S^{\de}.$$

We define

$$\wt W_d(X):=\bigsqcup\limits_{\emptyset\subset S\subset X^{\rm sing} \atop \wh \de_S\in B^{^{\geq0}}_{d}(\wh X_S)}W_{\de_S}^{+}(X_S),$$
where again $B^{^{\geq0}}_{d}(\wh X_S)$ is the set of strictly balanced multidegrees ${\widehat \de}_S\geq\un 0$ on $\wh X_S$ such that $|{\widehat \de}_S|=d$, and $\de_S={\widehat \de}_S{|_{X_S}}$ with $|\de_S|=d-\sharp S$. We will also use the notation
$$\de_S=(d^S_1,\ldots,d^S_\ga) $$ for the components of the multidegree $\de_S$ on $X_S$.
Let

$$\ov{A_d(X)}:=\ov{\bigcup_{\de\in B^{^{\geq0}}_d(X)}A_\de(X)}\subset \Pbar.$$

\begin{remark}\label{uno}
{\rm When the degree $d=1$, the elements of $B^{^{\geq0}}_1(X)$ are of the form $(d_1,\ldots,d_\ga)$ with $d_i=1$ for one suitable $i\in\{1,\ldots,\ga\}$ and $d_j=0$ for $j\neq i$. Thus, when we look at the set $B^{^{\geq0}}_{1}(\wh X_S)$, if it is nonempty it must be $\sharp S=1$.}
\end{remark}

We have the following result:

\begin{teo}\label{totalstable}
Let $X$ be a stable curve of genus $g\geq 2$ with $B^{\geq0}_1(X)\neq\emptyset$. Then
$$\wt W_1(X)=\ov{A_1(X)}\subset\ov{P_X^1}.$$
\end{teo}

\begin{proof}Inclusion ($\supset$) holds by lemma \ref{inclusion}. Now we prove inclusion ($\subset$).
By hypothesis the set $\ov{A_1(X)}$ is nonempty. We want to prove that for any $\emptyset\subset S\subset X^{\rm sing}$ such that $\wh\de_S\in{B^{^{\geq0}}_{1}(\wh X_S)}$, $$W_{\de_S}^{+}(X_S)\subset\ov{A_\de(X)}$$
for a certain $\de\in B^{^{\geq0}}_1(X)$. We can equivalently write
$$\wt W_1(X)=\left\{[M,S]\in\Pbar\mbox{ s.t. }M\in W_{\de_S}^{+}(X_S)\mbox{ with } \de_S\geq\un 0\right\}.$$

By remark \ref{uno} we can assume $\sharp S=1$, with $S=\{n\}$, so hereafter we will write $n$ instead of $S$. We take $M\in W_{\de_n}^{+}(X_n)$ with $\wh\de_n\in\ov{B^{^{\geq0}}_{1}(\wh X_n)}$, and consider $[M,n]$. Then by the stratification of $\ov{P_X^1}$ and remark \ref{uno} there exists $\de=(d_1,\ldots,d_\ga)\in B^{^{\geq0}}_d(X)$ such that $d_i-1=d^n_i=0$ for one $i$. We assume, with no loss of generality, that
$d_1-1=d^n_1=0$; now $M\in \Pic_{\de_n}X_n$ is such that $M=(\Oc_{X^n_1},\ldots,\Oc_{X^n_{r_n}})$, where $r_n$ is the number of connected components of $X_n$ and it is $1$ or $2$ whether $n$ is separating or not. Then, there exists $L\in\IM\al_X^{\de_n}$ such that the pullback of $L$ to $X_n$ is $M$, and of course $L=\Oc_X$. Let us take a moving point $p_u$ on $C_1\cap\dot X$ such that $p_u$ specializes to $n$ as $u\mapsto 0$. We see that by construction $\grad L(p_u)=\de$ and $L(p_u)\in\IM\al_X^{\de}$; moreover
$$L(p_u)\stackrel{u\rightarrow 0}{\longrightarrow}[M,n]\in\ov{A_\de(X)}.$$
\end{proof}

\begin{remark}
{\rm In \cite{bib:CapEst}[Proposition 3.15], the authors characterize the locus $\Sigma_g^1$ in $\ov{M}_g$ of the curves such that $B_1(X)$ is empty, i.e. the so called \textit{$1$-general} curves. They prove that if $g\geq 2$ and $g$ is odd, then the set $\Sigma_g^1$ is empty. Hence when $g$ is odd, if $B_1^{\geq0}(X)\subset B_1(X)$ is nonempty we are always in the case of theorem \ref{totalstable}. If, otherwise, $B_1^{\geq0}(X)=\emptyset$, then both the sets $\wt W_1(X)$ and $\ov{A_1(X)}$ are empty.}
\end{remark}


\begin{thebibliography}{alp}

\bibitem[A02]{bib:alexeev1}V. Alexeev, Complete moduli in the presence of semiabelian group action. Ann. of Math. (2) 155, 611–-708 (2002).
\bibitem[Al04]{bib:alexeev2}V. Alexeev, Compactified Jacobians and Torelli map. Publ. RIMS Kyoto Univ. 40, 1241–-1265 (2004).
\bibitem[AIK76]{AKI}A. Altman, A. Iarrobino, S. Kleiman, Irreducibility of the compactified
Jacobian. In Real and complex singularities (Proc. Ninth Nordic Summer School/NAVF Sympos. Math., Oslo, 1976), 1–12.
\bibitem[AK80]{bib:AltKl}A. Altman, S. Kleiman, Compactifying the Picard scheme, Adv. Math. 35, 50--112 (1980).
\bibitem[ACGH]{bib:ACGH}E. Arbarello, M. Cornalba, P.A. Griffiths, J. Harris, ``Geometry of Algebraic Curves'', Springer-Verlag New York Berlin Heidelberg Tokyo.
\bibitem[Br99]{bib:Bruno}A. Bruno, Degenerations of Linear Series and Binary Curves, Ph.D. Thesis, Brandeis University, (1999).
\bibitem[C1]{bib:compJac}L. Caporaso, A Compactification of the Universal Picard Variety over the Moduli Space of Stable Curves, J. of Amer. Math. Soc., Vol.7, No. 3, 589--660 (1994).
\bibitem[C2]{bib:theta}L. Caporaso, Geometry of the Theta Divisor of a compactified Jacobian, Journal of the European Mathematical Society 11, 1385--1427 (2009).
\bibitem[C3]{bib:survey}L. Caporaso, Compactified Jacobians, Abel maps and Theta divisors, Curves and Abelian varieties: international conference, in honor of Roy Smith's 65th birthday. March 30-April 2, 2007, University of Georgia, Athens. Contemporary Mathematics Volume 465. AMS Bookstore, (2008).
\bibitem[C4]{bib:naturality}L. Caporaso, Naturality of Abel maps, Manuscripta Mathematica, Vol 123, N.1, May 2007, 53--71.
    \bibitem[C5]{bib:binary}L. Caporaso, Brill Noether theory of binary curves, Mathematical Research Letters - Volume 17 - Issue 2/ March 2010, 243--262.
\bibitem[C6]{bib:linear}L. Caporaso, Linear series on semistable curves, (on line version) International Mathematics Research Notices.(2010), rnq188, 49 pages.
\bibitem[C7]{bib:Néron}L. Caporaso, Compactified Jacobians of N\'{e}ron type, Rendiconti Lincei: Matematica e Applicazioni 21, 1--15 (2010).
\bibitem[CE06]{bib:CapEst}L. Caporaso, E. Esteves, On Abel Maps of Stable Curves, Michigan Math. J.55, 575--607 (2007).
\bibitem[CCE08]{bib:gorenstein}L. Caporaso, J. Coelho, E. Esteves, Abel maps of Gorenstein curves, Rend. Circ. Mat. Palermo (2) 57, No. 1, 33--59 (2008).
\bibitem[C07]{bib:CoelTesi}J. Coelho, Abel maps for reducible curves, PhD Thesis, IMPA, Brasil (2007).
\bibitem[CP09]{bib:CoelPac}J. Coelho, M. Pacini, Abel maps for curves of compact type, J. Pure Appl. Algebra 214, No. 8, 1319--1333 (2010).
\bibitem[EGK00]{bib:EGK1}E. Esteves, M. Gagn\'{e}, S. Kleiman, Abel maps and presentation schemes, Comm.Algebra 28, 5961--5992 (2000).
\bibitem[EGK02]{bib:EGK2}E. Esteves, M. Gagn\'{e}, S. Kleiman, Autoduality of the compactified Jacobian, J. Lond. Math. Soc., II. Ser. 65, No. 3, 591--610 (2002).
\bibitem[EK05]{bib:EK}E. Esteves, S. Kleiman, The compactified Picard scheme of the compactified Jacobian, Adv.Math. 198, 484--503 (2005).
\bibitem[N10]{Ngo}B. C. Ng\^{o}, The Fundamental Lemma for Lie algebras. (Le Lemme Fondamental pour
les alg\'{e}bres de Lie.) (French) Publ. Math., Inst. Hautes \'{E}tud. Sci. 111, 1-271 (2010).

\end{thebibliography}
\end{document}